\documentclass[a4paper,11pt]{amsart}
\usepackage{pict2e}
\usepackage[T1]{fontenc}
\usepackage{amssymb}
\usepackage{tabularx}
\usepackage[ansinew]{inputenc}
\usepackage[english,
francais]{babel}
\usepackage{color}
\usepackage{xcolor}
\usepackage{graphicx}
 
\usepackage[active]{srcltx} 
\usepackage{amssymb,amsthm,amsfonts,amssymb,euscript,mathrsfs
,
}
\usepackage{amsxtra}
\usepackage{amscd}
\usepackage{float}



 \DeclareMathAlphabet{\got}{U}{euf}{m}{n}     
\DeclareMathAlphabet{\mat}{U}{msb}{m}{n}     
\DeclareMathAlphabet{\mathbold}{OML}{cmm}{bx}{it} 

\newtheorem{theo}{Th\'eor\`eme}[section]

\newtheorem{lem}{Lemme}[section]
\newtheorem{defi}{D\'efinition}[section]

\newtheorem{cor}{Corollaire}[section]
\newtheorem{rem}{Remarque}[section]

\newtheorem{Ex}{Exemple}[section]

\DeclareMathOperator{\Ad }{Ad }

\newcommand{\card}{\mathop{\mathrm{card}}}

\begin{document}
\title{sur l'indice des sous-alg\`ebres biparaboliques de $\mathfrak{gl}(n)$ }

\date{\today}
\selectlanguage{french}
\author{Meher Bouhani}
\address{ 
Universit\'{e} Tunis El-Manar\\
Facult\'{e} des Sciences de Tunis\\
D\'{e}partement de Math\'{e}matiques\\
Campus Universitaire\\
2092 El-Manar
 \\    Tunis, Tunisie}
\email{Meher.Bouhani@math.univ-poitiers.fr}

\maketitle

\selectlanguage{french} 
\begin{abstract}
On donne quelques propri\'et\'es sur l'indice des sous-alg\`ebres biparaboliques de l'alg\`ebre de Lie complexe $\mathfrak{gl}(n)$ qui permettent d'obtenir l'indice de certaines classes int\'eressantes de ces sous-alg\`ebres et d'en d\'eduire en particulier celles parmi elles qui sont des sous-alg\`ebres de Frobenius.  
\end{abstract}

\section{Introduction}
Soit $\mathfrak{g}$ l'alg\`ebre de Lie d'un groupe de Lie alg\'ebrique complexe $\mathbold{G}$ et $\mathfrak{g}^{\ast}$ son dual. Pour $f\in \mathfrak{g}^{*}$, on note  
$\mathfrak{g}_{f}$ le stabilisateur de $f$ pour l'action coadjointe. On appelle indice de $\mathfrak{g}$ et on note $\chi[\mathfrak{g}]$ la dimension minimale de $\mathfrak{g}_{f}$ lorsque $f$ parcourt $\mathfrak{g}^{\ast}$. Si $\chi[\mathfrak{g}]=0$, $\mathfrak{g}$ est dite une alg\`ebre de Frobenius.\\

Dans toute la suite, les groupes et les alg\`ebres de Lie consid\'er\'es sont alg\'ebriques d\'efinis sur le corps des complexes. Pour tous entiers naturels $r$ et $s$, on notera $r[s]$ le reste de la division euclidienne de $r$ par $s$ et $r \wedge s$ le plus grand commun diviseur de $r$ et $s$.\\

Pour toute paire $(\underline{a},\underline{b})$ {{\color{black} de
    compositions d'un entier $n\in\mathbb{N}^{\times}$},
  $\underline{a}=(a_{1},\ldots,a_{k})$ et
  $\underline{b}=(b_{1},\ldots,b_{t})$, on associe une unique
  sous-alg\`ebre biparabolique de $\mathfrak{gl}(n)$ (\`a conjugaison
  pr\`es) qu'on notera $\mathfrak{q}(\underline{a}\mid\underline{b})$,
  et toutes les sous-alg\`ebres biparaboliques de $\mathfrak{gl}(n)$
  sont ainsi obtenues (voir \cite{MB}). Si $\underline{b}=n$
  (resp. $\underline{a}=n$), $\mathfrak{q}(\underline{a}\mid n)$
  (resp. $\mathfrak{q}(n\mid\underline{b})$) est une sous-alg\`ebre
  parabolique de $\mathfrak{gl}(n)$, on la notera simplement
  $\mathfrak{p}(\underline{a})$
  (resp. $\overline{\mathfrak{p}}(\underline{b}))$. Soient
  $\underline{a}=(a_{1},\ldots,a_{k})\in\mathbb{N}^{k}$ et
  $\underline{b}=(b_{1},\ldots,b_{t})\in\mathbb{N}^{t}$ tels que
  $\sum_{1\leq l\leq k}a_{l}=\sum_{1\leq l\leq t}b_{l}=n$ et
  $\underline{a}^{'}$ (resp. $\underline{b}^{'}$) la composition de
  $n$ obtenue de $\underline{a}$ (resp. $\underline{b}$) en supprimant
  les termes nuls. Alors, on convient que
  $\mathfrak{p}(\underline{a})=\mathfrak{p}(\underline{a}^{'})$ et
  $\mathfrak{q}(\underline{a}\mid\underline{b})=\mathfrak{q}(\underline{a}^{'}\mid\underline{b}^{'})$.\\

Dans \cite{VCH}, les auteurs s'int\'eressent au cas des sous-alg\`ebres paraboliques de $\mathfrak{gl}(n)$ de la forme $\mathfrak{p}(a,\ldots,a,b)$, o\`u $a$ et $b$ sont deux entiers naturels non nuls, et montrent que la propri\'et\'e suivante :
$$ a\;est\;pair\;et\;a\wedge b=1\;\;\;(\ast)$$
est suffisante pour que $\mathfrak{p}_{s}(a,\ldots,a,b):=\mathfrak{p}(a,\ldots,a,b)\cap\mathfrak{sl}(n)$ soit une sous-alg\`ebre de Frobenius de $\mathfrak{sl}(n)$. Pour cette famille de sous-alg\`ebres paraboliques de $\mathfrak{gl}(n)$, on montre le th\'eor\`eme suivant (th\'eor\`eme \ref{thm5}) :

\begin{theo}
Soient $I=\lbrace(a,b)\in\mathbb{N^{\times}}^{2}\mid a\;est\;impair\;ou\;b\;est\;impair\rbrace$, $m\in\mathbb{N}^{\times}$ et $\phi_{m}$ la fonction d\'efinie sur $I$ par :
$$\phi_{m}(a,b)=\begin{cases} [\frac{m}{2}]+1\;si\;a\;impair\;et\;b\;impair\\
[\frac{m+1}{2}]\;si\;a\;impair\;et\;b\;pair\\
1\;si\;a\;pair\;et\;b\;impair\\
\end{cases}$$
Pour tous $a,b,m\in\mathbb{N^{\times}}$, on a :
$$\chi[\mathfrak{p}(\underbrace{a,\ldots,a}_{m},b)]=(a\wedge b)\phi_{m}(\frac{a}{a\wedge b},\frac{b}{a\wedge b})$$.
\end{theo}

Ainsi, on obtient une formule pour l'indice de $\mathfrak{p}(a,\ldots,a,b)$ et on en d\'eduit en particulier que, pour $n=ma+b$, $\mathfrak{p}_{s}(a,\ldots,a,b)$ est une sous-alg\`ebres de Frobenius de $\mathfrak{sl}(n)$ si et seulement si $a\wedge b=1$ et de plus l'une des conditions suivantes est v\'erifi\'ee
\begin{itemize}
\item[(i)]m=1
\item[(ii)]$a$ est pair, $b$ est impair
\item[(iii)]$a$ est impair, $b$ est pair et $m=2$
\end{itemize}
$\\$
Dans \cite{D.K}, Dergachev et Kirillov associent \`a chaque sous-alg\`ebre biparabolique $\mathfrak{q}(\underline{a},\underline{b})$ de $\mathfrak{gl}(n)$ un graphe appel\'e m\'eandre de $\mathfrak{q}(\underline{a},\underline{b})$, et au moyen de ce graphe, ils d\'ecrivent  l'indice de $\mathfrak{q}(\underline{a},\underline{b})$ (voir th\'eor\`eme \ref{thm1}). La description de $\chi[\mathfrak{q}(\underline{a}\mid\underline{b})]$ en fonction de $\underline{a}$ et $\underline{b}$ est une question int\'eressante. En effet, avoir une r\'eponse \`a cette question permettrait en particulier de d\'eterminer les sous-alg\`ebres biparaboliques de $\mathfrak{sl}(n)$ qui sont de Frobenius. Dans \cite{AE} et \cite{VCM}, les auteurs ont obtenu l'indice de $\mathfrak{p}(\underline{a})$ lorsque $k=2$ et $k=3$ :

\begin{theo}\cite{AE}\label{thm0}
$\chi[\mathfrak{p}(a_{1},a_{2})]=a_{1}\wedge a_{2}$
\end{theo}

\begin{theo}\cite{VCM}\label{thm0^{'}}
$\chi[\mathfrak{p}(a_{1},a_{2},a_{3})]=(a_{1}+a_{2})\wedge(a_{2}+a_{3})$
\end{theo}

Dans notre pr\'esent travail, on montre qu'on peut ramener la question
pr\'ec\'edente au cas des sous-alg\`ebres paraboliques de
$\mathfrak{gl}(2n)$ : en effet, \`a toute sous-alg\`ebre biparabolique
$\mathfrak{q}$ de $\mathfrak{gl}(n)$, on associe de mani\`ere
naturelle une sous-alg\`ebre parabolique $\mathfrak{p}$ de
$\mathfrak{gl}(2n)$ ayant le m\^eme indice (voir lemme
\ref{lem1}). {\color{black} Nous d\'emontrons ensuite, d'une mani\`ere
ind\'ependante, le th\'eor\`eme suivant qui est
une g\'en\'eralisation du lemme 4.2 de \cite{karnauhova} :}

\begin{theo}
Soient $\mathfrak{p}(a_{1},\ldots,a_{k})$ une sous-alg\`ebre
parabolique de $\mathfrak{gl}(n)$, $d_{k}=-(a_{1}+\cdots+a_{k-1})$ et
$d_{i}=(a_{1}+\cdots+a_{i-1})-(a_{i+1}+\cdots+a_{k})$, $ 1\leq i\leq
k-1$. Soit $ 1\leq i\leq k$.
\begin{itemize}
\item[1)] Supposons que $d_{i}\neq 0$. Pour tout $\alpha\in\mathbb{Z}$ tel que $a_{i}+\alpha|d_{i}|\geq 0$, on a 
$$\chi[\mathfrak{p}(a_{1},\ldots,a_{k})]=\chi[\mathfrak{p}(a_{1},\ldots,a_{i}+\alpha|d_{i}|,\ldots,a_{k})],\; 1\leq i\leq k$$ $\\$
En particulier, on a 
$$\chi[\mathfrak{p}(a_{1},\ldots,a_{k})]=\chi[\mathfrak{p}(a_{1},\ldots,a_{i-1},a_{i}[|d_{i}|],a_{i+1}\ldots,a_{k})],\; 1\leq i\leq k$$
\item[2)] Supposons que $d_{i}=0$. Alors, on a 
$$\chi[\mathfrak{p}(a_{1},\ldots,a_{k})]=a_{i}+\chi[\mathfrak{p}(a_{1},\ldots,a_{i-1},a_{i+1}\ldots,a_{k})]$$
\end{itemize}
\end{theo}

Ce th\'eor\`eme fournit un algorithme de r\'eduction permettant le calcul de l'indice d'une sous-alg\`ebre parabolique de $\mathfrak{gl}(n)$. De plus, il permet de red\'emontrer d'une mani\`ere simple les th\'eor\`emes \ref{thm0} et \ref{thm0^{'}} (voir th\'eor\`eme \ref{thm2}), d'obtenir l'indice de certaines classes des sous-alg\`ebres paraboliques de $\mathfrak{gl}(n)$ (voir th\'eor\`eme \ref{thm4}) et de donner en particulier de nouvelles familles de sous-alg\`ebres de Frobenius de $\mathfrak{gl}(n)$ (voir corollaire \ref{cor2} et lemme \ref{lem0}).

\section{R\'eduction}

Soit $\mathbold{G}$ un groupe de Lie alg\'ebrique complexe, $\mathfrak{g}$ son alg\`ebre
de Lie et $\mathfrak{g}^{\ast}$ le dual de $\mathfrak{g}$. Au moyen de la repr\'{e}sentation coadjointe, $\mathfrak{g}$ et $\mathbold{G}$ op\`{e}rent dans
$\mathfrak{g}^{*}$ par :
$$(x.f)(y)=f([y,x]), \;\; \forall x,y\in \mathfrak{g}\;\text{et}\;f\in
\mathfrak{g}^{*}$$
$$(x.f)(y)=f(\Ad x^{-1}y),\;\; \forall x\in\mathbold{G},y\in \mathfrak{g}\;\text{et}\;f\in
\mathfrak{g}^{*} $$ 
Pour $f\in \mathfrak{g}^{*}$, soit $\mathbold{G}_{f}$ le stabilisateur de $f$ pour cette action et $\mathfrak{g}_{f}$ son alg\`ebre de Lie:
$$\mathbold{G}_{f}=\{x\in \mathbold{G};f(\Ad x^{-1}y)=f(y),\;\; \forall y\in
\mathfrak{g}\}$$ 
  
$$\mathfrak{g}_{f}=\{x\in \mathfrak{g};f([x,y])=0,\;\; \forall y\in
\mathfrak{g}\}$$ 
On appelle indice de $\mathfrak{g}$ l'entier
$\chi[\mathfrak{g}]$ d\'{e}fini par:
$$\chi[\mathfrak{g}]=\min\{\dim \mathfrak{g}_{f}\;,\;f\in
\mathfrak{g}^{*}\}$$
Si $\chi[\mathfrak{g}]=0$, $\mathfrak{g}$ est dite une alg\`ebre de Frobenius.$\\$

Soient $\underline{a}=(a_{1},\dots, a_{k})$ et $\underline{b}=(b_{1},\ldots,b_{t})$ deux compositions d'un entier $n\in\mathbb{N}^{\times}$. Pour tout $1\leq i\leq k$ (resp. $1\leq j\leq t$), soient $\theta_{i}$ (resp. $\theta^{'}_{j}$) l'involution de $I_{i}=[a_{1}+\dots+a_{i-1}+1,a_{1}+\dots+a_{i-1}+a_{i}]\cap\mathbb{N}$ (resp. $J_{j}=[b_{1}+\dots+b_{j-1}+1,b_{1}+\dots+b_{j}]\cap\mathbb{N}$) d\'efinie par $\theta_{i}(x)=2(a_{1}+\dots+a_{i-1})+a_{i}-x+1,\;x\in I_{i}$ (resp. $\theta^{'}_{j}(y)=2(b_{1}+\dots+b_{j-1})+b_{j}-y+1,\;y\in J_{j}$). Soient $I=\cup_{1\leq i\leq k}I_{i}=\cup_{1\leq j\leq t}J_{j}=[1,n]\cap\mathbb{N}$. On associe \`a la paire $(\underline{a},\underline{b})$, les deux involutions $\theta_{\underline{a}}$ et $\theta_{\underline{b}}$ de $I$ telles que la restriction de $\theta_{\underline{a}}$ (resp. $\theta_{\underline{b}}$) \`a $I_{i}$ (resp. $J_{j}$) est \'egale \`a $\theta_{i}$ (resp. $\theta^{'}_{j}$), $1\leq i\leq k$ (resp. $1\leq j\leq t$). Soit $\mathbb{K}_{(\underline{a},\underline{b})}=<\theta_{\underline{a}},\theta_{\underline{b}}>$ le groupe engendr\'e par $\theta_{\underline{a}}$ et $\theta_{\underline{b}}$, il agit dans $I$.

On associe \`a la paire $(\underline{a},\underline{b})$ une sous-alg\`ebre biparabolique de $\mathfrak{gl}(n)$ not\'ee $\mathfrak{q}(\underline{a}\mid \underline{b})$, qui est unique \`a conjugaison pr\`es par le groupe $GL(n)$ (voir \cite{MB}) et un graphe not\'e $\Gamma(\underline{a}\mid\underline{b})$ et appel\'e m\'eandre associ\'e \`a $\mathfrak{q}(\underline{a}\mid \underline{b})$ ou m\'eandre de $\mathfrak{q}(\underline{a}\mid \underline{b})$, dont les sommets sont $n$ points cons\'{e}cutifs situ\'es sur une droite  horizontale D et num\'erot\'es $1,2,\ldots,n$. Il est construit de la mani\`ere suivante: on relie par un arc au dessous (resp. au dessus) de la droite D toute paire de sommets distincts de $\Gamma(\underline{a}\mid\underline{b})$ de la forme $(x,\theta_{\underline{a}}(x))$ (resp. $(x,\theta_{\underline{b}}(x))$), $x\in I$. 

$\\$
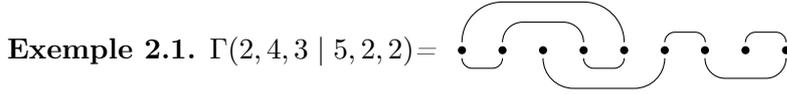
\begin{Ex}\label{Ex6.1} $\Gamma(2,4,3\mid5,2,2)$=
{\setlength{\unitlength}{0.021in}
\raisebox{-12\unitlength}{%
\begin{picture}(100,20)(-1,-11)
\multiput(0,3)(10,0){9}{\circle*{2}}
\put(20,5){\oval(40,20)[t]}
\put(20,5){\oval(20,10)[t]}
\put(55,5){\oval(10,5)[t]}
\put(75,5){\oval(10,5)[t]}
\put(5,1){\oval(10,5)[b]}
\put(35,1){\oval(30,15)[b]}
\put(35,1){\oval(10,5)[b]}
\put(70,1){\oval(20,10)[b]}
\end{picture}
} }
\end{Ex}

Si $\underline{b}=n$, $\mathfrak{q}(\underline{a}\mid n)$ est une sous-alg\`ebre parabolique de $\mathfrak{gl}(n)$ qu'on notera plus simplement $\mathfrak{p}(\underline{a})$, et on notera son m\'eandre $\Gamma(\underline{a})$.

\begin{defi}\cite{MB}
Si X est un sous-graphe de $\Gamma(\underline{a}\mid\underline{b})$, on note $S_{X}$ l'ensemble des sommets de $\Gamma(\underline{a}\mid\underline{b})$ appartenant \`a X.
\end{defi}

\begin{lem}\cite{MB}
Un sous-graphe connexe X de $\Gamma(\underline{a}\mid\underline{b})$ est une composante connexe si et seulement si $S_{X}$ est une $\mathbb{K}_{(\underline{a}\mid\underline{b})}$-orbite.
\end{lem}

\begin{defi}\cite{MB}
Une composante connexe X de $\Gamma(\underline{a}\mid\underline{b})$ est dite un cycle si tout $x\in S_{X}$ est distinct de $\theta_{\underline{a}}(x)$ et de $\theta_{\underline{b}}(x)$.\\
Une composante connexe X de $\Gamma(\underline{a}\mid\underline{b})$ est dite un segment si elle n'est pas un cycle $($i.e. il existe $x\in S_{X}$ v\'erifiant $\theta_{\underline{a}}(x)=x$ ou $\theta_{\underline{b}}(x)=x)$.
\end{defi}

\begin{theo}\label{thm1}\cite{D.K} Soient $\mathfrak{q}(\underline{a}\mid\underline{b})$ une sous-alg\`ebre biparabolique de $\mathfrak{gl}(n)$ et $\Gamma(\underline{a}\mid\underline{b})$ son m\'eandre, on a : 
$$\chi(\mathfrak{q}(\underline{a}\mid\underline{b}))=2\times (nombre\;de\;cycles)+\;nombre\;de\;segments$$
\end{theo}

\begin{rem}
Si $\mathfrak{q}(\underline{a}\mid\underline{b})$ est une
sous-alg\`ebre biparabolique de $\mathfrak{gl}(n)$, alors
$\mathfrak{q}_{s}(\underline{a}\mid\underline{b}):=\mathfrak{q}(\underline{a}\mid\underline{b})\cap\mathfrak{sl}(n)$
est une sous-alg\`ebre biparabolique de $\mathfrak{sl}(n)$ et
$\chi(\mathfrak{q}_{s}(\underline{a}\mid\underline{b}))=\chi(\mathfrak{q}(\underline{a}\mid\underline{b}))-1$.
\end{rem}

\begin{lem}\label{lem1}
Soient $\underline{a}=(a_{1},\ldots,a_{k})$ et
$\underline{b}=(b_{1},\ldots,b_{t})$ deux compositions de
n{\color{black}. Si on pose} $\underline{a}^{'}=(a_{k},\ldots,a_{1})$ et
$\underline{b}^{'}=(b_{t},\ldots,b_{1})$, on a :
\begin{enumerate}
\item[1)]$\chi[\mathfrak{q}(\underline{a}\mid\underline{b})]=\chi[\mathfrak{q}(\underline{b}\mid\underline{a})]$.
\item[2)]$\chi[\mathfrak{q}(\underline{a}\mid\underline{b})]=\chi[\mathfrak{q}(\underline{a}^{'}\mid\underline{b}^{'})]$.
\item[3)]$\chi[\mathfrak{q}(\underline{a}\mid\underline{b})]=\chi[\mathfrak{p}(a_{1},\dots,a_{k},b_{t},\dots,b_{1})]$.
\item[4)] Soit $t_{i}=(a_{1}+\cdots+a_{i})-(a_{i+1}+\cdots+a_{k}),\;1\leq i\leq k-1$, on a:\\
$\chi[\mathfrak{p}(a_{1},\ldots,a_{k})]=\begin{cases}\chi[\mathfrak{q}(a_{1},\ldots,a_{i}\mid a_{k},\ldots,a_{i+1},t_{i})]\;si\;t_{i}\geq 0\\
\chi[\mathfrak{q}(a_{1},\ldots,a_{i},-t_{i}\mid a_{k},\ldots,a_{i+1})]\;si\;t_{i}< 0
\end{cases}\\$ 
\item[5)] Supposons qu'il existe $1\leq i< k$ et $i+1<i^{'}\leq k$ tels que $a_{1}+\cdots+a_{i}=a_{i^{'}}+\cdots+a_{k}$, alors :
$\chi[\mathfrak{p}(\underline{a})]=\chi[\mathfrak{p}(a_{1},\ldots,a_{i},a_{i^{'}},\ldots,a_{k})]+\chi[\mathfrak{p}(a_{i+1},\ldots,a_{i^{'}-1})]$ .
\item[6)] Supposons qu'il existe $1\leq i<k$ et $1\leq j<t$ tels que $a_{1}+\cdots+a_{i}=b_{1}+\cdots+b_{j}$, alors :
$\chi[\mathfrak{q}(\underline{a}\mid\underline{b})]=\chi[\mathfrak{q}(a_{1},\ldots,a_{i}\mid b_{1},\ldots,b_{j})]+\chi[\mathfrak{q}(a_{i+1},\ldots,a_{k}\mid b_{j+1},\ldots,b_{t})]$ .
\end{enumerate} 
\end{lem}

\begin{proof}
\begin{enumerate}
\item[1)] Remarquons que $\mathbb{K}_{(\underline{a},\underline{b})}=\mathbb{K}_{(\underline{b},\underline{a})}$, par suite $\Gamma(\underline{b}\mid\underline{a})$ et $\Gamma(\underline{a}\mid\underline{b})$ ont les m\^emes composantes connexes. Le r\'esultat se d\'eduit du th\'eor\`eme \ref{thm1}.
\item[2)] Soit $\theta$ l'involution de $I$ d\'efinie par $\theta(x)=n-x+1,\;x\in I$. On a $\theta_{\underline{a}^{'}}=\theta\theta_{\underline{a}}\theta$ et $\theta_{\underline{b}^{'}}=\theta\theta_{\underline{b}}\theta$, par suite $\mathbb{K}_{(\underline{a}^{'},\underline{b}^{'})}$ est le conjugu\'e de $\mathbb{K}_{(\underline{a},\underline{b})}$ par $\theta$. De plus, un point $x\in I$ est invariant par $\theta_{\underline{a}}$ (resp. $\theta_{\underline{b}}$) si et seulement si $\theta(x)$ est invariant par $\theta_{\underline{a}^{'}}$ (resp. $\theta_{\underline{b}^{'}}$). D'o\`u $\Gamma(\underline{a}\mid\underline{b})$ et $\Gamma(\underline{a}^{'}\mid\underline{b}^{'})$ ont les m\^emes composantes connexes.
\item[3)] Soient
  $\underline{c}=(a_{1},\cdots,a_{k},b_{t},\cdots,b_{1})$,
  $I^{'}=[1,2n]\cap\mathbb{N}$ et $\theta$ l'involution de $I^{'}$
  d\'efinie par $\theta(x)=2n-x+1,\;x\in I^{'}$. Alors
  $\theta_{\underline{a}}$ (resp. $\theta_{\underline{b}}$) est la
  restriction \`a $I$ de $\theta_{\underline{c}}$
  (resp. $\theta\theta_{\underline{c}}\theta$). Par suite, prendre
  l'intersection avec $I$, induit une bijection de l'ensemble des
  $\mathbb{K}_{(\underline{c},2n)}$-orbites de $I^{'}$ sur l'ensemble
  des $\mathbb{K}_{(\underline{a},\underline{b})}$-orbites de $I$. De
  plus, un point $x\in I$ est invariant par $\theta_{\underline{b}}$
  si et seulement si $\theta(x)$ est invariant par
  $\theta_{\underline{c}}$. D'o\`u, il existe une bijection de
  l'ensemble des composantes connexes de
  $\Gamma(\underline{a}\mid\underline{b})$ sur celui des composantes
  connexes de $\Gamma(\underline{c}\mid 2n)$ qui conserve le nombre de
  cycles, ainsi que le nombre de segments. Le r\'esultat se d\'eduit
  du th\'eor\`eme \ref{thm1}.
\item[4)] Supposons $t_{i}\geq 0$, le r\'esultat se d\'emontre de la m\^eme mani\`ere que $3)$(voir figure 1).\\
Supposons $t_{i}\leq 0$, on a :
\begin{align*}
\chi[\mathfrak{p}(a_{1},\ldots,a_{k})]&=\chi[\mathfrak{p}(a_{k},\ldots,a_{1})]\\
&=\chi[\mathfrak{q}(a_{k},\ldots,a_{i+1}\mid a_{1},\ldots,a_{i},-t_{i})]\\
&=\chi[\mathfrak{q}(a_{1},\ldots,a_{i},-t_{i}\mid a_{k},\ldots,a_{i+1})]
\end{align*}
\item[5)] R\'esulte du fait que
  $\Gamma(\underline{a}\mid\underline{b})$ est r\'eunion disjointe
  {\color{black} des
    m\'eandres} $\Gamma(a_{1},\ldots,a_{i},a_{i^{'}},\ldots,a_{k})$ et
    $\Gamma(a_{i+1},\ldots,a_{i^{'}-1})$.
\item[6)]R\'esulte de 3) et 5).
\end{enumerate}

\end{proof}
$\\$
$\\$
$\\$
$\\$
$\\$
$\\$
$\\$
$\\$
$\\$
$\\$
$\\$
$\\$
\begin{figure}[htb]\label{fig10}
 
{\setlength{\unitlength}{0.021in}
\begin{center}
\begin{picture}(200,20)(-30,-80)

\put(79,0){\oval(10,10)[t]}
\put(150.5,-50){\oval(10,10)[t]}

\put(3.5,0){\oval(20,20)[b]}
\put(76.8,0){\oval(15,15)[b]}

\put(85,-50){\oval(20,20)[b]}
\put(148,-50){\oval(15,15)[b]}

\qbezier[140](-6,0)(86,0)(166,0)
\put(-34,-2){$\Gamma(\underline{a})=$}
\put(1,-15){$a_{1}$}
\put(74,-15){$a_{i}$}
\put(148,-15){$a_{k}$}
\put(91,-15){$a_{i+1}$}
\put(76,7.5){$t_{i}$}

\qbezier[75](75,-50)(95,-50)(155,-50)
\put(-33,-52){$\Gamma(a_{1},\ldots,a_{i}\mid a_{k},\ldots,a_{i+1},t_{i})=$}
\put(82,-66){$a_{1}$}
\put(146,-66){$a_{i}$}
\put(85,-36){$a_{k}$}
\put(147,-36){$t_{i}$}
\put(127,-36){$a_{i+1}$}

\linethickness{0.5mm}

\put(80,00){\oval(172.5,60)[t]}
\put(80.5,0){\oval(111.5,40)[t]}
\put(77.5,0){\oval(57,25)[t]}
\put(79.25,0){\oval(13.5,13)[t]}

\put(90,-50){\oval(30,20)[t]}
\put(134,-50){\oval(20,15)[t]}

\put(151,0){\oval(30,20)[b]}
\put(96,0){\oval(20,15)[b]}

\end{picture}
\end{center}
 }
\caption{}
\end{figure}

\begin{lem}\label{lem2}
\begin{itemize}

Soient $(a_{1},\ldots,a_{k})$ une composition de n, $d_{k}=-(a_{1}+\cdots+a_{k-1})$ et $d_{i}=(a_{1}+\cdots+a_{i-1})-(a_{i+1}+\cdots+a_{k})$, $ 1\leq i\leq k-1$. On a \\ 
\item[1)]
$\chi[\mathfrak{p}(a_{1},\ldots,a_{k})]=
\chi[\mathfrak{p}(a_{1},\ldots,a_{i},|a_{i}+d_{i}|,a_{i+1},\ldots,a_{k})]$, $ 1\leq i\leq k-1$.
\item[2)] Pour tout $1\leq i\leq k$ tel que $a_{i}+d_{i}\geq 0$, on a :
 $$\chi[\mathfrak{p}(a_{1},\ldots,a_{k})]=\chi[\mathfrak{p}(a_{1},\ldots,a_{i-1},a_{i}+d_{i},a_{i+1},\ldots,a_{k})]$$
\end{itemize}
\end{lem} 

\begin{proof}
\begin{enumerate}
\item[1)]
Soit $1\leq i\leq k-1$. Si $a_{i}+d_{i}\geq 0$, D'apr\`es le 4) puis le 3) du lemme \ref{lem1}, on a :
\begin{align*}
\chi[\mathfrak{p}(a_{1},\ldots,a_{k})]&=\chi[\mathfrak{q}(a_{1},\ldots,a_{i}\mid a_{k},\ldots,a_{i+1},a_{i}+d_{i})]\\
&=\chi[\mathfrak{p}(a_{1},\ldots,a_{i},a_{i}+d_{i},a_{i+1},\ldots,a_{k})]\\
\end{align*}
De m\^eme, si $a_{i}+d_{i}\leq 0$, on a :
\begin{align*}
\chi[\mathfrak{p}(a_{1},\ldots,a_{k})]&=\chi[\mathfrak{q}(a_{1},\ldots,a_{i},|a_{i}+d_{i}|\mid a_{k},\ldots,a_{i+1})]\\
&=\chi[\mathfrak{p}(a_{1},\ldots,a_{i},|a_{i}+d_{i}|,a_{i+1},\ldots,a_{k})]\\
\end{align*}
 \item[2)] Pour tout $2\leq i\leq k-1$ tel que $a_{i}+d_{i}\geq 0$, consid\'erons la composition $\underline{b}^{i}=(b^{i}_{1},\ldots,b^{i}_{k})$ o\`u $b^{i}_{j}=a_{j}$ si $j\neq i$ et $b^{i}_{i}=a_{i}+d_{i}$. Soit $T^{i}_{i-1}=(b^{i}_{1}+\dots+b^{i}_{i-2})-(b^{i}_{i}+\dots+b^{i}_{k}),\;2\leq i\leq k-1$. On v\'erifie que $|b^{i}_{i-1}+T^{i}_{i-1}|=a_{i},\;2\leq i\leq k-1$. Il r\'esulte de 1) que pour tout $2\leq i\leq k-1$, on a :
  \begin{align*}
\chi[\mathfrak{p}(b^{i}_{1},\ldots,b^{i}_{k})]&=\chi[\mathfrak{q}(b^{i}_{1},\ldots,b^{i}_{i-1},|b^{i}_{i-1}+T^{i}_{i-1}|,b^{i}_{i},\ldots,b^{i}_{k})]\\
&=\chi[\mathfrak{p}(a_{1},\ldots,a_{i},a_{i}+d_{i},a_{i+1},\ldots,a_{k})]\\
&=\chi[\mathfrak{p}(a_{1},\ldots,a_{k})]\\
\end{align*}
D'o\`u le r\'esultat pour $2\leq i\leq k-1$.\\
Supposons $a_{1}+d_{1}\geq 0$, il r\'esulte de lemme \ref{lem1} que :
\begin{align*}
\chi[\mathfrak{p}(a_{1},\ldots,a_{k})]&=\chi[\mathfrak{p}(a_{1},a_{1}+d_{1},a_{2},\ldots,a_{k})]\\
&=\chi[\mathfrak{q}(a_{1}\mid a_{k},\ldots,a_{2},a_{1}+d_{1})]\\
&=\chi[\mathfrak{p}(a_{k},\ldots,a_{2},a_{1}+d_{1})]\\
&=\chi[\mathfrak{p}(a_{1}+d_{1},a_{2},\ldots,a_{k})]\\
\end{align*}
D'o\`u le r\'esultat pour $i=1$.\\
Puisque $\chi[\mathfrak{p}(a_{1},\ldots,a_{k})]=\chi[\mathfrak{p}(a_{k},\ldots,a_{1})]$, on en d\'eduit alors le r\'esultat pour $i=k$. 
  \end{enumerate} 
\end{proof}

\begin{Ex}\label{Exemple1.1} Soit $\Gamma(5,1,2,4)$. Puisque $|5+1-4|=2$, alors  $\chi[\mathfrak{p}(5,1,2,4)]=\chi[\mathfrak{p}(5,1,4)]$. Aussi $|5-4|=1$, alors $\chi[\mathfrak{p}(5,1,4)]=\chi[\mathfrak{p}(5,4)]=1$.  
$\\$
$\\$
$\\$
$\\$
$\\$
$\Gamma(5,1,2,4)$=
{\setlength{\unitlength}{0.021in}
\raisebox{-12\unitlength}
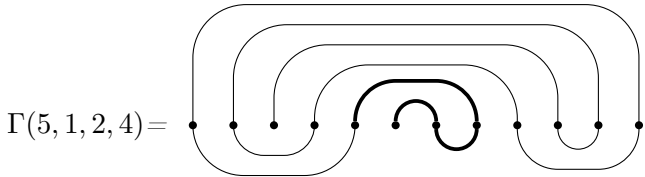
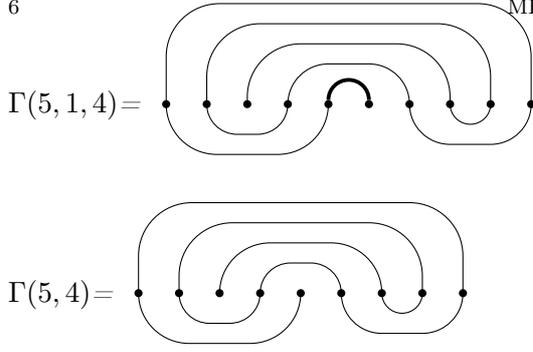
{%
\begin{picture}(100,20)(-1,-11)
\multiput(0,3)(10,0){12}{\circle*{2}}
\put(55,3){\oval(110,60)[t]}
\put(55,3){\oval(90,50)[t]}
\put(55,3){\oval(70,40)[t]}
\put(55,3){\oval(50,30)[t]}

\put(20,3){\oval(40,25)[b]}
\put(20,3){\oval(20,15)[b]}
\put(95,2){\oval(30,20)[b]}
\put(95,2){\oval(10,10)[b]}

\linethickness{0.5mm}
\put(55,4){\oval(30,20)[t]}
\put(55,4){\oval(10,10)[t]}
\put(65,2){\oval(10,10)[b]}
\end{picture}
} } 
$\\$
$\\$
$\\$
$\\$
$\\$
$\Gamma(5,1,4)$=
{\setlength{\unitlength}{0.021in}
\raisebox{-12\unitlength}{%
\begin{picture}(100,20)(-1,-11)
\multiput(0,3)(10,0){10}{\circle*{2}}

\put(45,3){\oval(90,50)[t]}
\put(45,3){\oval(70,40)[t]}
\put(45,3){\oval(50,30)[t]}
\put(45,3){\oval(30,20)[t]}

\put(20,3){\oval(40,25)[b]}
\put(20,3){\oval(20,15)[b]}
\put(75,3){\oval(30,20)[b]}
\put(75,3){\oval(10,10)[b]}

\linethickness{0.5mm}
\put(45,4){\oval(10,10)[t]}
\end{picture}
} } 
$\\$
$\\$
$\\$
$\\$
$\\$
$\Gamma(5,4)$=
{\setlength{\unitlength}{0.021in}
\raisebox{-12\unitlength}{%
\begin{picture}(100,20)(-1,-11)
\multiput(0,3)(10,0){9}{\circle*{2}}

\put(40,3){\oval(80,45)[t]}
\put(40,3){\oval(60,35)[t]}
\put(40,3){\oval(40,25)[t]}
\put(40,3){\oval(20,15)[t]}

\put(20,3){\oval(40,25)[b]}
\put(20,3){\oval(20,15)[b]}
\put(65,3){\oval(30,20)[b]}
\put(65,3){\oval(10,10)[b]}

\end{picture}
} } 
\end{Ex}
$\\$
$\\$

Le th\'eor\`eme suivant est une cons\'equense im\'ediate des lemmes \ref{lem1} et \ref{lem2}. 

\begin{theo}\label{cor1}
Soient $\mathfrak{p}(a_{1},\ldots,a_{k})$ une sous-alg\`ebre parabolique de $\mathfrak{gl}(n)$, $d_{k}=-(a_{1}+\cdots+a_{k-1})$ et $d_{i}=(a_{1}+\cdots+a_{i-1})-(a_{i+1}+\cdots+a_{k})$, $ 1\leq i\leq k-1$. Soit $ 1\leq i\leq k$.
\begin{itemize}
\item[1)] Supposons que $d_{i}\neq 0$. Pour tout $\alpha\in\mathbb{Z}$ tel que $a_{i}+\alpha|d_{i}|\geq 0$, on a 
$$\chi[\mathfrak{p}(a_{1},\ldots,a_{k})]=\chi[\mathfrak{p}(a_{1},\ldots,a_{i}+\alpha|d_{i}|,\ldots,a_{k})],\; 1\leq i\leq k$$ $\\$
En particulier, on a 
$$\chi[\mathfrak{p}(a_{1},\ldots,a_{k})]=\chi[\mathfrak{p}(a_{1},\ldots,a_{i-1},a_{i}[|d_{i}|],a_{i+1}\ldots,a_{k})],\; 1\leq i\leq k$$
\item[2)] Supposons que $d_{i}=0$. Alors, on a 
$$\chi[\mathfrak{p}(a_{1},\ldots,a_{k})]=a_{i}+\chi[\mathfrak{p}(a_{1},\ldots,a_{i-1},a_{i+1}\ldots,a_{k})]$$
\end{itemize}
\end{theo}

\begin{lem}\label{lemP}
Si $k\geq 2$ est un entier, il n'existe pas de k-uplet $(a_{1},\ldots,a_{k})$ de r\'eels v\'erifiant les in\'egalit\'es
$$0<a_{i}<|a_{1}+\cdots+a_{i-1}-(a_{i+1}+\cdots+a_{k})|,\;1\leq i\leq k$$
\end{lem}

\begin{proof}
On peut ramener la d\'emonstration au cas $a_{k}<a_{1}+\cdots+a_{k-1}$, en particulier l'ensemble $\Xi:=\{1\leq i\leq k-1\; tel\;que\;a_{1}+\cdots+a_{i}>a_{i+1}+\cdots+a_{k}\}$ est non vide ($k-1\in \Xi$). Soit $i_{0}=\inf\{i\in\Xi\}$, on voit alors que $i_{0}$ v\'erifie $a_{1}+\cdots+a_{i_{0}-1}\leq a_{i_{0}}+\cdots+a_{k}$ et $a_{1}+\cdots+a_{i_{0}}>a_{i_{0}+1}+\cdots+a_{k}$, en particulier $a_{i_{0}}\geq |a_{1}+\cdots+a_{i_{0}-1}-(a_{i_{0}+1}+\cdots+a_{k})|$. D'o\`u le r\'esultat.
\end{proof}

\begin{rem}\label{rem2} Il suit du lemme pr\'ec\'edent et du lemme \ref{lem1} que le th\'eor\`eme $\ref{cor1}$ donne un algorithme pour le calcul de l'indice des sous-alg\`ebres biparaboliques de $\mathfrak{gl}(n)$.
\end{rem} 

\begin{Ex}
\begin{align*}
\chi[\mathfrak{q}(111,13,79\mid 165,18,20)]&=\chi[\mathfrak{p}(111,13,79,20,18,165)]\\
&=\chi[\mathfrak{p}(111,13,20,18,165)]\\
&=\chi[\mathfrak{p}(111,13,20,18,3)]\\
&=\chi[\mathfrak{p}(3,13,20,18,3)]\\
&=\chi[\mathfrak{p}(3,3)]+\chi[\mathfrak{p}(13,20,18)]\\
&=(3\wedge 3)+((13+20)\wedge(20+18))\\
&=4\\
\end{align*}
\end{Ex}

\begin{theo}\label{thm2} Pour tous $a_{1},a_{2}\;et\;a_{3}\in\mathbb{N}^{\times}$, on a :

\begin{itemize}

\item[1)]  $\chi[\mathfrak{p}(a_{1},a_{2})]=a_{1}\wedge a_{2}$.
\item[2)]   $\chi[\mathfrak{p}(a_{1},a_{2},a_{3})]=(a_{1}+a_{2})\wedge (a_{2}+a_{3})$.
\end{itemize}

\end{theo}

\begin{proof}
\begin{itemize}

\item[1)]C'est clair d'apr\`es le th\'eor\`eme \ref{cor1}.

\item[2)] Soit $n=a_{1}+a_{2}+a_{3}$, le r\'esultat est vrai pour $n=3$ puisque $\chi[\mathfrak{p}(1,1,1)]=2$. Supposons que $n>3$, si $a_{1}=a_{3}$, le r\'esultat est vrai d'apr\`es le lemme \ref{lem1}, si $a_{1}\neq a_{3}$, il suit du lemme \ref{lemP} que les entiers naturels $a_{1},a_{2}$ et $a_{3}$ v\'erifient l'une des conditions suivantes :\\
i) $a_{1}\geq a_{2}+a_{3}$\\
ii) $a_{3}\geq a_{1}+ a_{2}$\\
iii) $a_{2}\geq|a_{1}-a_{3}|$

En appliquant le th\'eor\`eme \ref{cor1}, le r\'esultat s'en d\'eduit par r\'ecurrence sur $n$.
\end{itemize}
\end{proof}

\begin{cor}\label{cor2} Soient $a,b,c,\alpha\in\mathbb{N}^{\times}$, on a :
\begin{itemize}
\item[1)]$\chi[\mathfrak{p}(\alpha(a+b+c),a,b,c)]=\chi[\mathfrak{p}(a,\alpha|a-b-c|,b,c)]=\chi[\mathfrak{p}(a,b,\alpha|a+b-c|,c)]=\chi[\mathfrak{p}(a,b,c)]$
\item[2)] Supposons $(a+b)\wedge(b+c)=1$, les sous-alg\`ebres $\mathfrak{p}_{s}(a,b,c)$,$\mathfrak{p}_{s}(\alpha(a+b+c),a,b,c)$, $\mathfrak{p}_{s}(a,\alpha|a-b-c|,b,c)$, $\mathfrak{p}_{s}(a,b,\alpha|a+b-c|,c)$ et $\mathfrak{p}_{s}(a,b,c,\alpha(a+b+c))$ sont des sous-alg\`ebres de Frobenius de $\mathfrak{sl}(n)$.

\end{itemize}
\end{cor}

\begin{lem}\label{lem0} Pour tout $r\in\mathbb{N}^{\times}$ et $(\alpha_{1},\ldots,\alpha_{r})\in{\color{black}(\mathbb{N}^{\times})^{r}}$, on pose $a_{0}=1$, $a_{i}=\alpha_{i}(a_{0}+\dots+a_{i-1}),\;1\leq i\leq r$ et $n=a_{0}+\dots+a_{r}$. Alors $\mathfrak{p}_{s}(a_{0},\ldots,a_{r})$ est une sous-alg\`ebre de Frobenius de $\mathfrak{sl}(n)$.  
\end{lem}

\begin{proof}
Il suffit de remarquer que $a_{i}[a_{0}+\dots+a_{i-1}]=0,\;1\leq i\leq r$. Le r\'esultat d\'ecoule du th\'eor\`eme \ref{cor1}.
\end{proof}

\begin{lem}\label{lem25}
Soit $(a_{1},\cdots,a_{k})$ une composition d'un entier $n\in\mathbb{N}^{\times}$. Pour tout $\alpha\in\mathbb{N}^{\times}$, on a :
$$\chi[\mathfrak{p}(\alpha a_{1},\cdots,\alpha a_{k})]=\alpha\chi[\mathfrak{p}(a_{1},\cdots,a_{k})]$$
En particulier, $(a_{1}\wedge\ldots\wedge a_{k})$ divise $\chi[\mathfrak{p}(a_{1},\cdots,a_{k})]$.

\end{lem}

\begin{proof}
On a $\chi[\mathfrak{p}(\alpha)]=\alpha=\alpha\chi[\mathfrak{p}(1)]$, d'o\`u le r\'esultat pour $n=1$. Soient $d_{k}=-(a_{1}+\cdots+a_{k-1})$ et $d_{i}=(a_{1}+\cdots+a_{i-1})-(a_{i+1}+\cdots+a_{k})$, $1\leq i\leq k-1$. D'apr\`es le lemme \ref{lemP}, il existe $1\leq i\leq k$ v\'erifiant $a_{i}\geq|d_{i}|$. Supposons que $d_{i}=0$, alors $\chi[\mathfrak{p}(\alpha a_{1},\cdots,\alpha a_{k})]=\alpha a_{i}+\chi[\mathfrak{p}(\alpha a_{1},\cdots,\alpha a_{i-1},\alpha a_{i+1},\cdots,\alpha a_{k})]$, sinon, d'apr\`es le th\'eor\`eme \ref{cor1}, on a $\chi[\mathfrak{p}(\alpha a_{1},\cdots,\alpha a_{k})]=\chi[\mathfrak{p}(\alpha a_{1},\cdots,\alpha a_{i-1},\alpha(a_{i}[|d_{i}|]),\alpha a_{i+1},\cdots,\alpha a_{k})]$
 
Le r\'esultat s'ensuit alors par r\'ecurrence sur $n$. 
\end{proof}

\begin{lem}\label{lem0^{'}}
Pour tous $a,k\in\mathbb{N}$ tel que $a>1$, on a :
$$\chi[\mathfrak{p}(1,a,a^{2},\cdots,a^{k})]=\frac{a^{[\frac{k}{2}]+1}-1}{a-1}$$
\end{lem}

\begin{proof}
On a $\chi[\mathfrak{p}(1)]=\chi[\mathfrak{p}(1,a)]=1$, si bien que le r\'esultat est vraie pour $k=0$ et $k=1$. Supposons que $k\geq 2$, d'apr\`es le th\'eor\`eme \ref{cor1}, on a :
\begin{align*}
\chi[\mathfrak{p}(1,a,a^{2},\cdots,a^{k})]&=\chi[\mathfrak{p}(1,a,a^{2},\cdots,a^{k-1},1)]\\
&=1+\chi[\mathfrak{p}(a,a^{2},\cdots,a^{k-1})]\\
&=1+a\chi[\mathfrak{p}(1,a,a^{2},\cdots,a^{k-2})]\\
\end{align*}
Le r\'esultat s'obtient alors par r\'ecurrence sur $k$.
\end{proof}

\begin{lem}\label{lem3}
Soient $(a_{1},\ldots,a_{k})$ une composition de n, $d_{k}=-(a_{1}+\cdots+a_{k-1})$ et $d_{i}=(a_{1}+\cdots+a_{i-1})-(a_{i+1}+\cdots+a_{k})$, $1\leq i\leq k$.

\begin{itemize}
 \item[1)] Supposons $a_{i}+d_{i}\geq 0$ pour $2\leq i\leq k-1$, alors :
$$\chi[\mathfrak{p}(a_{1},\ldots,a_{k})]=\chi[\mathfrak{p}(a_{i}+d_{i},a_{i+1},\ldots,a_{k},a_{2},\ldots,a_{i})]$$

\item[2)] Supposons $a_{1}+d_{1}\leq 0$, alors :

$$\chi[\mathfrak{p}(a_{1},\ldots,a_{k})]=\chi[\mathfrak{p}(a_{2},\ldots,a_{k},|a_{1}+d_{1}|)]$$

\item[3)] Supposons $a_{k}+d_{k}\leq 0$, alors :

$$\chi[\mathfrak{p}(a_{1},\ldots,a_{k})]=\chi[\mathfrak{p}(|a_{k}+d_{k}|,a_{1},\ldots,a_{k-1})]$$
\end{itemize}
\end{lem}

\begin{proof}

\begin{itemize}
 \item[1)] D'apr\`es le lemme \ref{lem1}, on a :
 \begin{align*} \chi[\mathfrak{p}(a_{1},\ldots,a_{k})]&=\chi[\mathfrak{q}(a_{1},\ldots,a_{i}\mid a_{k},\ldots,a_{i+1},a_{i}+d_{i})]\\
&=\chi[\mathfrak{q}(a_{i}+d_{i},a_{i+1},\ldots,a_{k}\mid a_{i},\ldots,a_{1})]\\
&=\chi[\mathfrak{p}(a_{i}+d_{i},a_{i+1},\ldots,a_{k},a_{1},\ldots,a_{i})]\\
\end{align*}
Il r\'esulte du th\'eor\`eme \ref{cor1} que :
$$\chi[\mathfrak{p}(a_{i}+d_{i},a_{i+1},\ldots,a_{k},a_{1},\ldots,a_{i})]=
\chi[\mathfrak{p}(a_{i}+d_{i},a_{i+1},\ldots,a_{k},a_{2},\ldots,a_{i})]$$

 \item[2)] D'apr\`es le lemme \ref{lem1}, on a :
 \begin{align*}\chi[\mathfrak{p}(a_{1},\ldots,a_{k})]&=\chi[\mathfrak{q}(a_{1},|a_{1}+d_{1}|\mid a_{k},\ldots,a_{2})]\\
&=\chi[\mathfrak{q}(|a_{1}+d_{1}|,a_{1}\mid a_{2},\ldots,a_{k})]\\
&=\chi[\mathfrak{p}(|a_{1}+d_{1}|,a_{1},a_{k},a_{k-1},\ldots,a_{2})]\\
\end{align*}
Il r\'esulte du th\'eor\`eme \ref{cor1} que :
$$\chi[\mathfrak{p}(|a_{1}+d_{1}|,a_{1},a_{k},a_{k-1},\ldots,a_{2})]=\chi[\mathfrak{p}(|a_{1}+d_{1}|,a_{k},a_{k-1},\ldots,a_{2})]$$
Enfin, on a :
$$\chi[\mathfrak{p}(|a_{1}+d_{1}|,a_{k},a_{k-1},\ldots,a_{2})]=\chi[\mathfrak{p}[(a_{2},\ldots,a_{k},||a_{1}+d_{1}|)]$$

\item[3)] R\'esulte de 2) et du fait que $\chi[\mathfrak{p}(a_{1},\ldots,a_{k})]=\chi[\mathfrak{p}(a_{k},\ldots,a_{1})]$.
\end{itemize}

\end{proof}

\begin{cor}\label{cor3}
 Soient $\underline{a}=(a_{1},\ldots,a_{k})$, $\underline{b}=(b_{1},\ldots,b_{t})$ deux compositions de n, $\tilde{a}_{i}=a_{i}+\cdots+a_{k}$, $1<i\leq k$, et $\tilde{b}_{i}=b_{1}+\cdots+b_{i}$, $1\leq i\leq t$. On a :\\
 \begin{enumerate}
\item[1)]Pour tout $1<i\leq k$,
{\color{black} $$
\chi[\mathfrak{p}(\underline{a})]=\begin{cases}\chi[\mathfrak{p}(a_{1}-2\tilde{a}_{i},a_{i},\ldots,a_{k},a_{2},\ldots,a_{i-1})]\;si\;a_{1}\geq 2\tilde{a}_{i}\\
\chi[\mathfrak{p}(a_{i},\ldots,a_{k},a_{2},\ldots,a_{i-1},2\tilde{a}_{i}-a_{1})]\;si\;a_{1}<2\tilde{a}_{i}\\
 \end{cases}$$}
 \item[2)]Pour tout $1\leq i\leq t$,
{\color{black} $$\chi[\mathfrak{q}(\underline{a}\mid\underline{b})]=\begin{cases}\chi[\mathfrak{q}(a_{1}-2\tilde{b}_{i},b_{i},\ldots,b_{1},a_{2},\ldots,a_{k}\mid b_{i+1},\ldots,b_{t})]\;si\;a_{1}\geq 2\tilde{b}_{i}\\
\chi[\mathfrak{q}(b_{i},\cdots,b_{1},a_{2},\ldots,a_{k}\mid 2\tilde{b}_{i}-a_{1},b_{i+1},\ldots,b_{t})]\;si\;a_{1}<2\tilde{b}_{i}\\
 \end{cases}$$}
\end{enumerate}
\end{cor}

\begin{proof}

D'apr\`es le th\'eor\`eme \ref{cor1} et le 1) du lemme \ref{lem3}, on a :
\begin{align*}
\chi[\mathfrak{p}(\underline{a})]&=\chi[\mathfrak{p}(n,a_{2},\ldots,a_{k})]\\
&=\chi[\mathfrak{p}(2n-a_{1}-2\tilde{a}_{i},a_{i},\cdots ,a_{k},a_{2},\ldots,a_{i-1})]\\
\end{align*}
\begin{enumerate}
\item[1)]
Supposons $a_{1}\geq 2\tilde{a}_{i}$, alors $2n-a_{1}-2\tilde{a}_{i}\geq 2(n-a_{1})$. D'apr\`es le th\'eor\`eme \ref{cor1}, on a :
\begin{align*}
\chi[\mathfrak{p}(\underline{a})]&=\chi[\mathfrak{p}(2n-a_{1}-2\tilde{a}_{i}-2(n-a_{1}),a_{i},\ldots,a_{k},a_{2},\ldots,a_{i-1})]\\
&=\chi[\mathfrak{p}(a_{1}-2\tilde{a}_{i},a_{i},\ldots,a_{k},a_{2},\ldots,a_{i-1})]\\
\end{align*}

Supposons $a_{1}< 2\tilde{a}_{i}$, on distingue deux cas :

Si $n\geq 2\tilde{a}_{i}$, il suit du th\'eor\`eme \ref{cor1}, puis du lemme \ref{lem3} que l'on a :

\begin{align*}
\chi[\mathfrak{p}(\underline{a})]&=\chi[\mathfrak{p}(2n-a_{1}-2\tilde{a}_{i},a_{i},\ldots,a_{k},a_{2},\ldots,a_{i-1})]\\
&=\chi[\mathfrak{p}(2n-a_{1}-2\tilde{a}_{i}-(n-a_{1}),a_{i},\ldots,a_{k},a_{2},\ldots,a_{i-1})]\\
&=\chi[\mathfrak{p}(n-2\tilde{a}_{i},a_{i},\ldots,a_{k},a_{2},\ldots,a_{i-1})]\\
&=\chi[\mathfrak{p}(a_{i},\ldots,a_{k},a_{2},\ldots,a_{i-1},(n-a_{1})-(n-2\tilde{a}_{i}))]\\
&=\chi[\mathfrak{p}(a_{i},\ldots,a_{k},a_{2},\ldots,a_{i-1},2\tilde{a}_{i}-a_{1})]\\
\end{align*}

Si $n<2\tilde{a}_{i}$, il suit du lemme \ref{lem3}, puis du th\'eor\`eme \ref{cor1} que l'on a :

\begin{align*}
\chi[\mathfrak{p}(\underline{a})]&=\chi[\mathfrak{p}(2n-a_{1}-2\tilde{a}_{i},a_{i},\ldots,a_{k},a_{2},\ldots,a_{i-1})]\\
&=\chi[\mathfrak{p}(a_{i},\ldots,a_{k},a_{2},\ldots,a_{i-1},(n-a_{1})-(2n-a_{1}-2\tilde{a}_{i}))]\\
&=\chi[\mathfrak{p}(a_{i},\ldots,a_{k},a_{2},\ldots,a_{i-1},2\tilde{a}_{i}-n)]\\
&=\chi[\mathfrak{p}(a_{i},\ldots,a_{k},a_{2},\ldots,a_{i-1},2\tilde{a}_{i}-n+(n-a_{1}))]\\
&=\chi[\mathfrak{p}(a_{i},\ldots,a_{k},a_{2},\ldots,a_{i-1},2\tilde{a}_{i}-a_{1})]\\
\end{align*}
\item[2)]C'est une cons\'equense im\'ediate de 1) et du lemme \ref{lem1}.

\end{enumerate}
\end{proof}

\begin{theo}\label{thm3}
 Soient a et b deux entiers naturels non nuls, on a :
 
 $$\chi[\mathfrak{p}(a,a,a,b)]=\chi[\mathfrak{p}(a\wedge b,a,a)]=(a+(a\wedge b))\wedge (a-(a\wedge b))$$
 \end{theo}
\begin{proof}
 Comme $\chi[\mathfrak{p}(a,a,a,b)]=\chi[\mathfrak{p}(b,a,a,a)]$, il suit du lemme \ref{lem3} que l'on a :\\
$\chi[\mathfrak{p}(a,a,a,b)]=\chi[\mathfrak{p}(a,a,a,b[a])]$. On peut se ramener alors au cas $b<a$.\\
Mais alors, appliquant le lemme \ref{lem3} et le r\'esultat pr\'ec\'edent, on a :
\begin{align*} 
\chi[\mathfrak{p}(a,a,a,b)]&=\chi[\mathfrak{p}(3a-b,a,a,a)]\\
&=\chi[\mathfrak{p}(a-b,a,a,a)]\\
&=\chi[\mathfrak{p}(a,a,a,a-b)].\\
\end{align*}
On peut supposer alors que $2b<a$.

D'apr\`es le th\'eor\`eme \ref{cor1}, on a : 

\begin{align*}
 \chi[\mathfrak{p}(a,a,a,b)]&=\chi[\mathfrak{p}(a,a[b],a,b)]\\
&=\chi[\mathfrak{p}(a,a[b],b-a[b],b)]\\
&=\chi[\mathfrak{p}(a[2b],a[b],b-a[b],b)]\\
&=\chi[\mathfrak{p}(a[2b],a[b],a[2b],b)]\\
&=\chi[\mathfrak{p}(a[2b],a[2b],a[2b],b)]\\
 \end{align*}
 Supposons $a=1$, alors $\chi[\mathfrak{p}(1,1,1,b)]=\chi[\mathfrak{p}(1,1,1,b[3])]$. Or, $\chi[\mathfrak{p}(1,1,1)]=\chi[\mathfrak{p}(1,1,1,1)]=\chi[\mathfrak{p}(1,1,1,2)]=2$. Par suite, le r\'esultat est vrai pour $a=1$.\\
D'autre part on a : 
 
$(a[2b]-(a\wedge b))\wedge (a[2b]+(a\wedge b))=(a-(a\wedge b))\wedge (a+(a\wedge b))$

Le r\'esultat s'obtient alors par r\'ecurrence sur $a$ .
 
 \end{proof}

 \begin{cor}\label{cor4}
  
  \begin{itemize}
 \item[(1)] $\chi[\mathfrak{p}(a,a,a,b)]=\begin{cases} 2(a\wedge b)\;si\;a\;est\;impair \\
   (a\wedge b)\;si\;a\;est\;pair\;et\;b\;est\;impair \\                                      
  \end{cases}$
 \item[(2)] $\chi[\mathfrak{p}(a,a,a,b)]=1$ si et seulement si a est pair et $a\wedge b=1$ 
 \end{itemize}
 \end{cor}

 \begin{theo}\label{thm4} Soient $a,b,c,d\in\mathbb{N}^{\times}$, on a :
 
 \begin{enumerate}
 \item[1)]Supposons $(a+b+c)\wedge(b+c+d)\geq b+c$, alors :
 $$\chi[\mathfrak{p}(a,b,c,d)]=[(a+b+c)\wedge(b+c+d)]-(b+c)+(b\wedge c)$$
 \item[2)]Supposons $(a+b+c)\wedge(c+d)\geq a+b$, alors :
 $$\chi[\mathfrak{p}(a,b,c,d)]=[(a+b+c)\wedge(c+d)]-(a+b)+(a\wedge b)$$
\item[3)]Supposons $(a+b)\wedge(b+c+d)\geq c+d$, alors :
 $$\chi[\mathfrak{p}(a,b,c,d)]=[(a+b)\wedge(b+c+d)]-(c+d)+(c\wedge d)$$
 \item[4)]Supposons $a>d$ et $(b+a-d)\wedge(b+c)\geq a-d$, alors :
 $$\chi[\mathfrak{p}(a,b,c,d)]=[(b+a-d)\wedge(b+c)]-(a-d)+(a\wedge d)$$
 \end{enumerate}
 \end{theo}

 \begin{proof}
 \begin{enumerate}
  \item[1)] soient $a^{'}=a+b+c,\;d^{'}=b+c+d$ et $p=a^{'}\wedge d^{'}$. On a 
  $$\chi[\mathfrak{p}(a,b,c,d)]=\chi[\mathfrak{p}(a^{'}-(b+c),b,c,d^{'}-(b+c))]$$
Supposons qu'il existe $\alpha\in\mathbb{N}$ tel que $a^{'}-\alpha d^{'}\geq b+c$, il r\'esulte du th\'eor\`eme \ref{cor1} que 
     $$\chi[\mathfrak{p}(a,b,c,d)]=\chi[\mathfrak{p}(a^{'}-\alpha d^{'}-(b+c),b,c,d^{'}-(b+c))]$$
  De m\^eme, s'il existe $\alpha\in\mathbb{N}$ tel que $d^{'}-\alpha a^{'}\geq b+c$, on a 
  $$\chi[\mathfrak{p}(a,b,c,d)]=\chi[\mathfrak{p}(a^{'}-(b+c),b,c,d^{'}-\alpha a^{'}-(b+c))]$$
  Or, par hypoth\`se, on a $p\geq b+c$. Ainsi, si on applique successivement le 1) du th\'eor\`eme \ref{cor1}, on se ram\`ene au cas :  
  $$\chi[\mathfrak{p}(a,b,c,d)]=\chi[\mathfrak{p}(p-(b+c),b,c,p-(b+c))]$$
 Maintenant, si on utilise le 5) du lemme \ref{lem1}, on obtient
   $$\chi[\mathfrak{p}(a,b,c,d)]=\chi[\mathfrak{p}(p-(b+c),p-(b+c))]+\chi[\mathfrak{p}(b,c)]$$
 Il suit alors du th\'eor\`eme \ref{thm2} que l'on a :
 $$\chi[\mathfrak{p}(a,b,c,d)]=p-(b+c)+(b\wedge c)$$
 D'o\`u le r\'esultat.
  \item[2)] Comme $\chi[\mathfrak{p}(a,b,c,d)]=\chi[\mathfrak{p}(d,c,b,a)]$, il suit du lemme \ref{lem3} que si $c+d\geq a+b$, on a :
  $$\chi[\mathfrak{p}(a,b,c,d)]=\chi[\mathfrak{p}(d+c-a-b,b,a,c)]$$
  Mais alors, si $(d+c)\wedge (a+b+c)\geq a+b$, il r\'esulte du 1) que l'on a :
  $$\chi[\mathfrak{p}(a,b,c,d)]=p-(a+b)+(a\wedge b)$$
  D'o\`u le r\'esultat.
  \item[3)] R\'esulte de 2) et du fait que $\chi[\mathfrak{p}(a,b,c,d)]=\chi[\mathfrak{p}(d,c,b,a)]$.
  
 \item[4)]Soit $p=(b+a-d)\wedge(b+c)\geq a-d\geq 0$, la r\'eduction du th\'eor\`eme \ref{cor1} donne :
 \begin{align*}
 \chi[\mathfrak{p}(a,b,c,d)]&=\chi[\mathfrak{p}(a,b,c+b+(a-d),d)]\\
 &=\chi[\mathfrak{p}(a,(b+a-d)-(a-d),c+b+(a-d),d)]\\
 &=\chi[\mathfrak{p}(a,p-(a-d),(a-d),d)]\\
 &=p-(a-d)+\chi[\mathfrak{p}(a,(a-d),d)]\\
 &=p-(a-d)+(a\wedge d)\\
 \end{align*}
 
 \end{enumerate}
 \end{proof}
 
 \begin{lem}\label{lem4}
 Soient $(a_{1},\ldots,a_{k})$ une composition de $n$ et $(b_{1},\ldots,b_{t})$ une composition de $n^{'}$ avec $n\geq n'$. 
 \begin{itemize}
 \item[1)] Supposons que $a_{1}> a_{k}$, on a :
 $$\chi[\mathfrak{p}(a_{1},\ldots,a_{k})]=\chi[\mathfrak{p}((a_{1}-a_{k})-a_{1}[a_{1}-a_{k}],a_{1}[a_{1}-a_{k}],a_{2},\ldots,a_{k-1})]$$
\item[2)] Il existe une composition $(d_{1},\ldots,d_{s})$ de
  $n-n^{'}$ et $\alpha\in\mathbb{N}$ telle que,
  {\color{black}$pour\;tout\;(c_{1},\ldots,c_{m})\in\mathbb{N}^{m}$,}
$$\chi[\mathfrak{p}(a_{1},\ldots,a_{k},c_{1},\ldots,c_{m},b_{t},\ldots,b_{1})]= \alpha+\chi[\mathfrak{p}(d_{1},\ldots,d_{s},c_{1},\ldots,c_{m})]$$.
 \end{itemize}  
 \end{lem}
 \begin{proof}
 \begin{itemize}
 \item[1)] D'apr\`es le 3) du lemme \ref{lem3}, on a : 
{\color{black}$$\chi[\mathfrak{p}(a_{1},\ldots,a_{k})]=
\chi[\mathfrak{p}(n-2a_{k},a_{1},\ldots,a_{k-1})]$$}
 D'apr\`es le th\'eor\`eme \ref{cor1}, on a :
 \begin{align*}
 \chi[\mathfrak{p}(a_{1},\ldots,a_{k})]&=\chi[\mathfrak{p}(n-2a_{k},a_{1},\ldots,a_{k-1})]\\
 &=\chi[\mathfrak{p}(n-2a_{k},a_{1}[a_{1}-a_{k}],a_{2}\ldots,a_{k-1})]\\
 &=\chi[\mathfrak{p}((a_{1}-a_{k})-a_{1}[a_{1}-a_{k}],a_{1}[a_{1}-a_{k}],a_{2}\ldots,a_{k-1})]\\
 \end{align*}
 \item[2)] D'apr\`es 1), on a : \\
{\color{black}\begin{multline*}
 \chi[\mathfrak{p}(a_{1},\dots,a_{k},c_{1},\ldots,c_{m},b_{t},\dots,b_{1})]=\\
  \begin{cases}a_{1}+\chi[\mathfrak{p}(a_{2},\dots,a_{k},c_{1},\ldots,c_{m},b_{t},\dots,b_{2})]\;si\;a_{1}=b_{1}\\
 \chi[\mathfrak{p}((a_{1}-b_{1})-a_{1}[a_{1}-b_{1}],a_{1}[a_{1}-b_{1}],a_{2},\dots,a_{k},c_{1},\ldots,c_{m},b_{t},\dots,b_{2})]\;si\;a_{1}>b_{1}\\
 \chi[\mathfrak{p}((a_{2},\dots,a_{k},c_{1},\ldots,c_{m},b_{t},\dots,b_{2},b_{1}[b_{1}-a_{1}],(b_{1}-a_{1})-b_{1}[b_{1}-a_{1}])]\;si\;b_{1}>a_{1}\\
 \end{cases}
\end{multline*}}
Le r\'esultat s'ensuit par r\'ecurrence sur l'entier $n^{'}$.  
\end{itemize}
\end{proof}

\begin{lem}\label{lem5}
Soient $(a,b_{1},\ldots,b_{t})\in(\mathbb{N^{*}})^{t+1}$ et $b=b_{1}+\cdots+b_{t}$. On a :
\begin{itemize}
\item[1)] Si $a\geq 2b$, alors :
$\chi[\mathfrak{p}(b_{1},\ldots,b_{t},\underbrace{a-2b,\ldots,a-2b}_{r},\underbrace{a,\ldots,a}_{n},\underbrace{a-2b,\ldots,a-2b}_{r})]=\\ \chi[\mathfrak{p}(b_{1},\ldots,b_{t},\underbrace{a-2b,\ldots,a-2b}_{n+2r})]$, pour tout $(r,n)\in \mathbb{N}^{2}$.
\item[2)] Si $b< a< 2b$, alors :
$\chi[\mathfrak{p}(b_{1},\ldots,b_{t},\underbrace{a,\ldots,a}_{n})]=\chi[\mathfrak{p}(b_{t},\ldots,b_{1},\underbrace{2b-a,\ldots,2b-a}_{n})]$, pour tout $n\in \mathbb{N}$.
\end{itemize}
 \end{lem}

\begin{proof}
\begin{itemize}
\item[1)]
Il est clair que le r\'esultat est vrai pour $n=0$. Pour $n=1$, le r\'esultat d\'ecoule du th\'eor\`eme \ref{cor1}. Supposons que $n\geq 2$. \\
D'apr\`es le lemme \ref{lem4}, il existe une composition $(d_{1},\ldots,d_{s})$ de $b$ et un entier $\alpha\geq 0$ tels que pour tout $(a_{1},\dots,a_{n})\in\mathbb{N}^{n}$, on a :
{\color{black}\begin{multline*}
\chi[\mathfrak{p}(b_{1},\ldots,b_{t},\underbrace{a-2b,\ldots,a-2b}_{r},a_{1},\ldots,a_{n},\underbrace{a-2b,\ldots,a-2b}_{r})]=\\
\alpha+\chi[\mathfrak{p}(d_{1},\ldots,d_{s},a_{1},\ldots,a_{n})]
\end{multline*}} 
En particulier, on a 
{\color{black}\begin{multline*}
\chi[\mathfrak{p}(b_{1},\ldots,b_{t},\underbrace{a-2b,\ldots,a-2b}_{r},\underbrace{a,\ldots,a}_{n},\underbrace{a-2b,\ldots,a-2b}_{r})]=
\\\alpha+\chi[\mathfrak{p}(d_{1},\ldots,d_{s},\underbrace{a,\ldots,a}_{n})]
\end{multline*}} 
D'apr\`es le corollaire \ref{cor3}, on a :\\
\begin{align*}
\chi[\mathfrak{p}(d_{1},\ldots,d_{s},\underbrace{a,\ldots,a}_{n})]&=\chi[\mathfrak{p}(a,\underbrace{a,\ldots,a}_{n-1},d_{s},\ldots,d_{1})]\\
&=\chi[\mathfrak{p}(a-2b,,d_{s},\ldots,d_{1},\underbrace{a,\ldots,a}_{n-1})]\\
&=\chi[\mathfrak{p}(\underbrace{a,\ldots,a}_{n-1},d_{1},\ldots,d_{s},a-2b)]\\
\end{align*}

D'autre part on a : $2(d_{1}+\cdots+d_{s}+a-2b)=2(a-b)\geq a$. Il suit du corollaire \ref{cor3}, que l'on a :\\  
$\chi[\mathfrak{p}(\underbrace{a,\ldots,a}_{n-1},d_{1},\ldots,d_{s},a-2b)]=\chi[\mathfrak{p}(d_{1},\ldots,d_{s},a-2b,\underbrace{a,\ldots,a}_{n-2},a-2b)]$. D'o\`u :\\
$\chi[\mathfrak{p}(b_{1},\ldots,b_{t},\underbrace{a-2b,\ldots,a-2b}_{r},\underbrace{a,\ldots,a}_{n},\underbrace{a-2b,\ldots,a-2b}_{r})]=\\ \chi[\mathfrak{p}(b_{1},\ldots,b_{t},\underbrace{a-2b,\ldots,a-2b}_{r+1},\underbrace{a,\ldots,a}_{n-2},\underbrace{a-2b,\ldots,a-2b}_{r+1})]$. Le r\'esultat s'ensuit par r\'ecurrence sur l'entier $n$.\\

\item[2)]
D'apr\`es le lemme \ref{lem4}, il existe {\color{black} une} composition $(d_{1},\ldots,d_{s})\in\mathbb{N}^{s}$ de $a-b$ et un entier $\alpha\geq 0$ tels que pour tout $(a_{1},\dots,a_{n-1})\in\mathbb{N}^{n-1}$, on a : \\ 
$\chi[\mathfrak{p}(b_{1},\ldots,b_{t},a_{1},\ldots,a_{n-1},a)]=\alpha+\chi[\mathfrak{p}(d_{1},\ldots,d_{s},a_{1},\ldots,a_{n-1})]$\\
En particulier, on a \\
$\chi[\mathfrak{p}(b_{1},\ldots,b_{t},\underbrace{a,\ldots,a}_{n})]=\alpha+\chi[\mathfrak{p}(d_{1},\ldots,d_{s},\underbrace{a,\ldots,a}_{n-1})]$.\\

Puisque $a> 2(a-b)$, on d\'eduit alors de 1) que :\\
\begin{align*}
\chi[\mathfrak{p}(b_{1},\ldots,b_{t},\underbrace{a,\ldots,a}_{n})]&=\alpha+\chi[\mathfrak{p}(d_{1},\ldots,d_{s},\underbrace{a,\ldots,a}_{n-1})]\\
&=\alpha+\chi[\mathfrak{p}(d_{1},\ldots,d_{s},\underbrace{2b-a,\ldots,2b-a}_{n-1})]\\
&=\chi[\mathfrak{p}(b_{1},\ldots,b_{t},\underbrace{2b-a,\ldots,2b-a}_{n-1},a)]\\
\end{align*}

Il suit du corollaire \ref{cor3} que l'on a :\\
\begin{align*}
\chi[\mathfrak{p}(b_{1},\ldots,b_{t},\underbrace{a,\ldots,a}_{n})]&=\chi[\mathfrak{p}(a,\underbrace{2b-a,\ldots,2b-a}_{n-1},b_{t},\ldots,b_{1})]\\
&=\chi[\mathfrak{p}(b_{t},\ldots,b_{1},\underbrace{2b-a,\ldots,2b-a}_{n})]\\
\end{align*}

D'o\`u le r\'esultat.
\end{itemize}
\end{proof}

\begin{theo}\label{thm5}
Soient $I=\lbrace(a,b)\in\mathbb{N^{\times}}^{2}\mid a\;est\;impair\;ou\;b\;est\;impair\rbrace$, $m\in\mathbb{N}^{\times}$ et $\phi_{m}$ la fonction d\'efinie sur $I$ par :
$$\phi_{m}(a,b)=\begin{cases} [\frac{m}{2}]+1\;si\;a\;impair\;et\;b\;impair\\
[\frac{m+1}{2}]\;si\;a\;impair\;et\;b\;pair\\
1\;si\;a\;pair\;et\;b\;impair\\
\end{cases}$$
Pour tous $a,b,m\in\mathbb{N^{\times}}$, on a :
$$\chi[\mathfrak{p}(\underbrace{a,\ldots,a}_{m},b)]=(a\wedge b)\phi_{m}(\frac{a}{a\wedge b},\frac{b}{a\wedge b})$$.
\end{theo}

\begin{proof}
Si $m=1$, le r\'esultat est vrai puisque $\phi_{1}(\frac{a}{a\wedge b},\frac{b}{a\wedge b})=1$. Supposons  $m>1$.\\
Supposons $a=1$, on a :\\
 $\chi[\mathfrak{p}(\underbrace{1,\ldots,1}_{m},b)]=\chi[\mathfrak{p}(\underbrace{1,\ldots,1}_{m},b[m])]=\chi[\mathfrak{p}(\underbrace{1,\ldots,1}_{b[m]})]+\chi[\mathfrak{p}(\underbrace{1,\ldots,1}_{m-b[m]})]=[\frac{b[m]+1}{2}]+[\frac{m-b[m]+1}{2}]=\phi_{m}(1,b[m])=\phi_{m}(1,b)$. D'o\`u le r\'esultat pour $a=1$.

D'apr\`es le lemme \ref{lem3}, on a :
$$\chi[\mathfrak{p}(\underbrace{a,\ldots,a}_{m},b)]=\begin{cases}\chi[\mathfrak{p}(\underbrace{a,\ldots,a}_{m},b[a])]\;si\;m\;est\;impair\\
\chi[\mathfrak{p}(\underbrace{a,\ldots,a}_{m},b[2a])]\;si\;m\;est\;pair\\
\end{cases}$$

et $\chi[\mathfrak{p}(\underbrace{a,\ldots,a}_{m},b)]=\chi[\mathfrak{p}(\underbrace{a,\ldots,a}_{m},ma-b)]$ si $b\leq ma$. Par suite, on peut se ramener au cas o\`u $b\leq a$.

D'apr\`es le lemme \ref{lem5}, on a :
$$\chi[\mathfrak{p}(\underbrace{a,\ldots,a}_{m},b)]=\begin{cases}\chi[\mathfrak{p}(\underbrace{a-2b,\ldots,a-2b}_{m},b)]\;si\;a\geq 2b\\
\chi[\mathfrak{p}(\underbrace{2b-a,\ldots,2b-a}_{m},b)]\;si\;a<2b\\
\end{cases}$$

Supposons que $a=2b$, alors $\chi[\mathfrak{p}(\underbrace{a,\ldots,a}_{m},b)]=b$. Le r\'esultat est vrai vue que $\phi_{m}(\frac{a}{a\wedge b},\frac{b}{a\wedge b})=\phi_{m}(2,1)=1$. Supposons que $a\neq 2b$, on v\'erifie que $\phi_{m}(\frac{a}{a\wedge b},\frac{b}{a\wedge b})=\phi_{m}(\frac{|a-2b|}{|a-2b|\wedge b},\frac{b}{|a-2b|\wedge b})$. Le r\'esultat s'ensuit alors par r\'ecurrence sur $a$.
\end{proof}

 \begin{cor}\label{cor5}
 $\chi[\mathfrak{p}(\underbrace{a,\ldots,a}_{m},b)]=1$ si et seulement si $a\wedge b=1$ et de plus l'une des conditions suivantes est v\'erifi\'ee
\begin{itemize}
\item[(i)]m=1
\item[(ii)]$a$ est pair, $b$ est impair
\item[(iii)]$a$ est impair, $b$ est pair et $m=2$
\end{itemize}
 \end{cor}
 
 \section{quelques propri\'et\'es sur le m\'eandre}
 
 Soient $\underline{a}=(a_{1},\ldots,a_{k})$ une composition de $n$, $\mathfrak{p}(\underline{a})$ une sous-alg\`ebre parabolique de $\mathfrak{gl}(n)$ et $\Gamma(\underline{a})$ son m\'eandre. Si X est un cycle de $\Gamma(\underline{a})$, il existe un entier $s>1$, appel\'e dimension de X et $k\in\mathbb{N}^{\times}$ tels que les sommets de X sont de la forme $x_{1}<x_{1}+s-1<x_{2}<\ldots <x_{k}<x_{k}+s-1$ (voir \cite{MB}). Si $s>2$, pour tous $1\leq i\leq k$ et $x_{i}\leq y\leq x_{i}+s-1$, on dit que le cycle (resp.  segment) Y, tel que $y\in S_{Y}$, est \`a l'int\'{e}rieur de X. Le cycle X est dit maximal s'il n'est pas \`a l'int\'{e}rieur d'un autre cycle, auquel cas, sa dimension est donn\'ee par $\dim(X)=2\card$\{cycles se trouvant \`{a} l'int\'{e}rieur\}$+\card$\{segments se trouvant \`{a} l'int\'{e}rieur\}$+2$ (voir \cite{A.O}) . Par convention, un segment X qui n'est pas contenu dans un cycle est consid\'{e}r\'{e} comme un cycle maximal de dimension 1. Il suit du th\'eor\`eme \ref{thm1} que l'indice de $\mathfrak{p}(\underline{a})$ est \'egal  \`a la somme des dimensions des cycles maximaux de $\Gamma(\underline{a})$.
    
\begin{rem}\label{rr}
\begin{itemize}   
\item[i)] Un sommet $x$ est l'extr\'emit\'e d'un segment de $\Gamma(\underline{a})$ si et seulement si, il existe $1\leq i\leq k$ tels que $a_{i}$ est impair et $x=a_{1}+\dots+a_{i-1}+[\frac{a_{i}+1}{2}]$ ou $n$ est impair et $x=\frac{n+1}{2}$.
\item[ii)] Un cycle maximal de $\Gamma(\underline{a})$ de dimension 1
  peut \^etre r\'eduit \`a {\color{black} un} point $x$. On v\'erifie
  dans ce cas, que $n$ est un entier impair, et de plus, il existe
  $1\leq i\leq k$ tels que $a_{i}=1$ et
  $x=a_{1}+\dots+a_{i}=\frac{n+1}{2}$.
\end{itemize}
\end{rem}  
  
\begin{defi} Soient X un cycle maximal de $\Gamma(\underline{a})$ de dimension $s>1$ et $x_{1}<x_{1}+s-1<x_{2}<\ldots <x_{k}<x_{k}+s-1$ les sommets de X. On appelle bout de X un arc de X joignant deux sommets de la forme $x_{i}$ et $x_{i}+s-1$.
\end{defi}

\begin{lem}\cite{MB}\label{lem11}
Tout cycle maximal de $\Gamma(\underline{a})$ de dimension $s>1$ a exactement deux bouts.
\end{lem}

\begin{defi}
On dira que deux m\'eandres $\Gamma_{1}$ et $\Gamma_{2}$ sont \'equivalents s'il existe une bijection de l'ensemble des composantes connexes de $\Gamma_{1}$ sur l'ensemble des composantes connexes de $\Gamma_{2}$ qui conserve le nombre de cycles maximaux, ainsi que leurs dimensions.
\end{defi}

Compte tenu des r\'eductions utilis\'ees dans la preuve du lemme \ref{lem1} dont d\'ecoule le th\'eor\`eme \ref{cor1}, on d\'eduit le lemme suivant :

\begin{lem}\label{lem10}
Soit $(a_{1},\ldots,a_{k})$ une composition, $d_{k}=-(a_{1}+\cdots+a_{k-1})$ et $d_{i}=(a_{1}+\cdots+a_{i-1})-(a_{i+1}+\cdots+a_{k})$, $ 1\leq i\leq k-1$. Pour tout $1\leq i\leq k$ tel que $d_{i}\neq 0$, les m\'eandres $\Gamma(a_{1},\ldots,a_{k})$ et $\Gamma(a_{1},\ldots,a_{i-1},a_{i}[|d_{i}|],a_{i+1},\ldots,a_{k})$ sont \'equivalents.
\end{lem}

\begin{lem}
\begin{enumerate}
 \item[1)] Pour tous $a_{1},a_{2}\in\mathbb{N}^{\times}$, le m\'eandre $\Gamma(a_{1},a_{2})$ contient un seul cycle maximal de dimension $a_{1}\wedge a_{2}$.
 \item[2)] Pour tous $a_{1},a_{2},a_{3}\in\mathbb{N}^{\times}$, le m\'eandre $\Gamma(a_{1},a_{2},a_{3})$ contient deux cycles maximaux de dimensions respectives $p-a_{2}[p]$ et $a_{2}[p]$, o\`u $p=(a_{1}+a_{2})\wedge (a_{2}+a_{3})$.
 \item[3)]Soit $a,b,m\in\mathbb{N}^{\times}$ et $\phi_{m}$ la fonction d\'efinie au th\'eor\`eme \ref{thm5}, le m\'eandre  $\Gamma(\underbrace{a,\ldots,a}_{m},b)$ contient $\phi_{m}(\frac{a}{a\wedge b},\frac{b}{a\wedge  b})$ cycles maximaux, chaque cycle est de dimension $(a\wedge b)$.
 \end{enumerate}
\end{lem}

\begin{proof}
\begin{itemize}
\item[1)] Le m\'eandre $\Gamma(1,1)$ est un segment, d'o\`u le r\'esultat pour $a_{1}=a_{2}=1$. Il suit du lemme \ref{lem10} qu'un raisonnement par r\'ecurrence sur la somme $a_{1}+a_{2}$ donne le r\'esultat.
 
\item[2)] Le m\'eandre $\Gamma(1,1,1)$ est compos\'e de deux segments, d'o\`u le r\'esultat pour $a_{1}=a_{2}=a_{3}=1$. D'autre part, le m\'eandre $\Gamma(a_{1},a_{2},a_{1})$ est compos\'e de deux cycles maximaux de dimensions respectives $a_{1}$ et $a_{2}$, d'o\`u le r\'esultat pour $a_{1}=a_{3}$. De m\^eme, il suit du lemme \ref{lem10} que le r\'esultat s'obtient par r\'ecurrence sur la somme $a_{1}+a_{2}+a_{3}$ . 
\item[3)]Le m\'eandre $\Gamma(\underbrace{1,\ldots,1}_{m+1})$ est compos\'e de $[\frac{m+2}{2}]$ segments, d'o\`u le r\'esultat pour $a=b=1$. D'autre part, on a vu dans la preuve du th\'eor\`eme \ref{thm5} qu'on peut se ramener au cas o\`u $b\leq a$ et que $$\chi[\mathfrak{p}(\underbrace{a,\ldots,a}_{m},b)]=\begin{cases}\chi[\mathfrak{p}(\underbrace{a-2b,\ldots,a-2b}_{m},b)]\;si\;a\geq 2b\\
\chi[\mathfrak{p}(\underbrace{2b-a,\ldots,2b-a}_{m},b)]\;si\;a<2b\\
\end{cases}$$
Comme toutes les r\'eductions utilis\'ees dans cette preuve d\'ecoulent du th\'eor\`eme \ref{cor1}, il suit du lemme \ref{lem10} que si $a\geq 2b$ (resp. $a<2b$), le m\'eandre $\Gamma(\underbrace{a,\ldots,a}_{m},b)$ et le m\'eandre $\Gamma(\underbrace{a-2b,\ldots,a-2b}_{m},b)$ (resp. $\Gamma(\underbrace{2b-a,\ldots,2b-a}_{m},b)$) sont \'equivalents. Le r\'esultat s'ensuit par r\'ecurrence sur $a$.

\end{itemize}
\end{proof}

\begin{rem}
Les assertions 1) et 2) du lemme pr\'ec\'edent sont donn\'es d'une mani\`ere ind\'ependante dans \cite{VCM2}.
\end{rem}

\begin{lem}
\begin{enumerate}
\item[1)]Soit $r$ le nombre de cycles maximaux de $\Gamma(a_{1},\ldots,a_{k})$, on a :
 $$r\leq[\frac{k+1}{2}]$$
\item[2)]$\Gamma(a_{1},\ldots,a_{k})$ contient un seul cycle maximal si est seulement si $\chi[\mathfrak{p}(a_{1},\ldots,a_{k})]=a_{1}\wedge a_{2}\wedge \ldots \wedge a_{k}$.
\end{enumerate}
\end{lem} 
\begin{proof}
\begin{enumerate}
\item[1)] On pose $n=a_{1}+\dots+a_{k}$,
  $a_{k+1}=-n$,\\ $I_{i}=[a_{1}+\dots+a_{i-1}+1,a_{1}+\dots+a_{i}],\;1\leq
  i\leq k$ et $I_{k+1}=[1,n]$. Soit $1\leq i\leq k+1$ tel que
  $|a_{i}|>1$. Supposons qu'il existe deux cycles maximaux distincts X
  et Y de $\Gamma(a_{1},\ldots,a_{k})$, $x\in I_{i}\cap S_{X}$ et
  $y\in I_{i}\cap S_{Y}$ tels ques $x$ et
  $2(a_{1}+\dots+a_{i-1})+a_{i}-x+1$ sont li\'es par un bout de X, $y$
  et $2(a_{1}+\dots+a_{i-1})+a_{i}-y+1$ sont li\'es par un bout de
  Y. On voit alors que, soit le sommet $x$ est compris entre les
  sommets $y$ et $2(a_{1}+\dots+a_{i-1})+a_{i}-y+1$, soit le sommet
  $y$ est compris entre les sommets $x$ et
  $2(a_{1}+\dots+a_{i-1})+a_{i}-x+1$. Il s'ensuit que l'un des deux
  cycles X et Y est \`a l'int\'{e}rieur de l'autre, ce qui
  impossible. Il en r\'esulte que pour tout $1\leq i\leq k+1$, il
  existe au plus un sommet $ x\leq
  {\color{black}a_{1}+\cdots+a_{i-1}+}[\frac{a_{i}}{2}]$ appartenant \`a
  $I_{i}$ tel que $x$ et $2(a_{1}+\dots+a_{i-1})+a_{i}-x+1$ sont
  li\'es par un bout d'un cycle maximal de
  $\Gamma(a_{1},\ldots,a_{k})$. Compte tenu de la remarque \ref{rr} et
  du lemme \ref{lem11}, on d\'eduit alors que le nombre de cycles
  maximaux de $\Gamma(a_{1},\ldots,a_{k})$ est major\'e par
  $[\frac{k+1}{2}]$.

\item[2)] Pour tout $n\in\mathbb{N}^{\times}$, le m\'eandre $\Gamma(\underbrace{1,\ldots,1}_{n})$ est compos\'e de $[\frac{n+1}{2}]$ segments, o\`u chaque segment est consid\'er\'e un cycle maximal de dimension $1$, et le m\'eandre $\Gamma(\underbrace{\alpha,\ldots,\alpha}_{n})$ est compos\'e de $[\frac{n+1}{2}]$ cycles maximaux, o\`u chaque cycle est de dimension $\alpha$.\\
 Soit $1\leq i\leq k$, supposons que $d_{i}=0$ (voir les d\'efinitions des $d_{i}$ dans l'\'enonc\'e du lemme \ref{lem10}), le {\color{black}m\'eandre} $\Gamma(a_{1},\ldots,a_{k})$ est compos\'e d'un cycle maximal de dimension $a_{i}$ et du m\'eandre $\Gamma(a_{1},\ldots,a_{i-1},a_{i+1},\ldots,a_{k})$, alors que le m\'eandre $\Gamma(\alpha a_{1},\ldots,\alpha a_{k})$ est compos\'e d'un cycle maximal de dimension $\alpha a_{i}$ et du m\'eandre $\Gamma(\alpha a_{1},\ldots,\alpha a_{i-1},\alpha a_{i+1},\ldots,\alpha a_{k})$. Supposons que $d_{i}\neq 0$, il r\'esulte du lemme \ref{lem10} que les m\'eandres $\Gamma(a_{1},\ldots,a_{k})$ et\\
 $\Gamma(a_{1},\ldots,a_{i-1},a_{i}[|d_{i}|],a_{i+1},\ldots,a_{k})$ sont \'equivalents, aussi que les m\'eandres \\
 $\Gamma(\alpha a_{1},\ldots,\alpha a_{k})$ et $\Gamma(\alpha a_{1},\ldots,\alpha a_{i-1},\alpha (a_{i}[|d_{i}|]),\alpha a_{i+1},\ldots,\alpha a_{k})$ sont \'equivalents. Il s'ensuit qu'un raisonnement par r\'ecurrence sur la somme $a_{1}+\dots+a_{k}$ implique que les m\'eandres $\Gamma(a_{1},\ldots,a_{k})$ et
 $\Gamma(\alpha a_{1},\ldots,\alpha a_{k})$ ont le m\^eme nombre de cycles maximaux, appelons-les respectivement $X_{1},\ldots,X_{e}$ et $\tilde{X}_{1},\ldots,\tilde{X}_{e}$, et qu'ils peuvent \^etre num\'erot\'es de sorte que $\dim(\tilde{X}_{i})=\alpha\dim(X_{i}),\;1\leq i\leq e$. On en d\'eduit alors que le PGCD de $a_{1},\ldots,a_{k}$ divise la dimension de chaque cycle maximal de $\Gamma(a_{1},\ldots,a_{k})$. En particulier, si $\chi[\mathfrak{p}(a_{1},\ldots,a_{k})]=a_{1}\wedge\ldots\wedge a_{k}$, alors $\Gamma(a_{1},\ldots,a_{k})$ contient un seul cycle maximal.
  
R\'eciproquement, supposons que $\Gamma(a_{1},\ldots,a_{k})$ contient un seul cycle maximal de dimension $s$, alors $s$ divise $a_{i}$, $ 1\leq i\leq k$, d'o\`u $s$ divise le PGCD de $a_{1},\ldots,a_{k}$, et par suite $\chi[\mathfrak{p}(a_{1},\ldots,a_{k})]=s=a_{1}\wedge\ldots\wedge a_{k}$. 

\end{enumerate}
\end{proof}
 
 \bibliographystyle{smfplain} 
\bibliography{biblio}

\end{document}